\documentclass[12pt]{amsart}
\usepackage{amsmath,amssymb,amscd,amsthm}
\usepackage[all,cmtip]{xy} 
\usepackage{cite}

\usepackage{geometry} 
\usepackage[all]{xypic}
\usepackage{graphicx}
\usepackage{mathrsfs}

\newtheorem{thm}{Theorem}[section]
\newtheorem{cor}[thm]{Corollary}
\newtheorem{lem}[thm]{Lemma}

\theoremstyle{definition}
\newtheorem{defn}[thm]{Definition}
\theoremstyle{remark}

\numberwithin{equation}{section}

\newcommand{\smp}[1]{\left(#1 \right)}
\newcommand{\midp}[1]{\left[#1 \right]}
\newcommand{\bigp}[1]{\left\{#1 \right\}}

\newcommand{\C}{\mathbb{C}} 
\newcommand{\R}{\mathbb{R}}

\DeclareMathOperator{\Sym}{Sym}
\DeclareMathOperator{\Herm}{Herm}

\newcommand{\ppb}{\partial \bar{\partial}}

\providecommand{\abs}[1]{\left| #1\right|}
\providecommand{\norm}[1]{\left\lVert#1\right\rVert}

\usepackage{esint}

\title{A Liouville Theorem for the Complex Monge-Amp\`ere Equation}
\author{Yu Wang}
\address{Department of Mathematics, Columbia University,  Room 509, MC 4406,
2990 Broadway
\\ New York, NY 10027, USA}
\email{yuwang@math.columbia.edu}

\begin{document}

\begin{abstract}
In this note, we derive a Liouville theorem for the complex Monge-Amp\`ere equation from the small perturbation result of O. Savin \cite{Savin-smp}. Let $dx$ stands for the standard Lebesgue measure, our result states that if a plurisubharmonic function $u$ solves 
\[
\smp{ i\ppb u}^n  = dx , \quad \text{ on }  \C^n 
\]
and $u$ satisfies the growth condition
\[ 
u = \frac{1}{2} \abs{x}^2 + o (\abs{x}^2 ) , \quad \text{ as } x\rightarrow \infty,
\]
then $u$ differs from $\abs{x}^2/2$ by a linear function.
\end{abstract}

\maketitle


\section{Introduction}
In this note, we consider the global solutions of the complex Monge-Amp\`ere equation. Denote the Lebesgue measure by $dx$, our result states:

\begin{thm}
\label{Thm Main}
If the plurisubharmonic function $u$ is a viscosity solution of 
\begin{equation}
\label{eq CMA}
\smp{ i\ppb u}^n  = dx , \quad \text{ on }  \C^n 
\end{equation}
and $u$ satisfies the growth condition
\begin{equation}
\label{cond growth}
u = \frac{1}{2} \abs{x}^2 + o (\abs{x}^2), \quad \text{ as } x\rightarrow \infty,
\end{equation}
then 
\begin{equation}
u =  \frac{1}{2} \abs{x}^2 + l(x)
\end{equation}
where $l(x)$ is a linear function.
\end{thm}

By a linear change of coordinates, one can replace $\abs{x}^2/2$ by every quadratic polynomial $P$ such that 
\[
\smp{i \ppb P}^n = dx. 
\]

We have stated the theorem in terms of viscosity solutions for our convenience. It is also valid for pluripotential solutions, as viscosity and pluripotential solutions are equivalent (see \cite{EGZ,YuWang1}).

Unlike the real Monge-Amp\`ere equation, global solutions of \eqref{eq CMA} cannot be classified without any restriction on solution's growth at infinity. Consider the following example due to Blocki \cite{Blocki}: the function
\[
u = \abs{z} (1 + \abs{w}^2)
\]
satisfies \eqref{eq CMA} on $\C^2$ in viscosity sense. However $u$ is clearly not the pull back of a quadratic polynomial by a holomorphic mapping. In fact, $u$ is not even $C^2$ at the points $\{(z,w) : z = 0\}$. We also notice that, along the diagonal direction $z= w$ 
\[
u ( x) \sim  \abs{x}^3 , \quad \text{ as } \quad x = (z,w) \rightarrow \infty. 
\]

A disadvantage of Theorem \ref{Thm Main} (and our proof) is that we have not been able to handle the case $u - \abs{x}^2/2$ has exactly quadratic growth, i.e., 
\[
 \frac{\abs{x}^2}{2C }\leq u - \abs{x}^2/2 \leq \frac{C}{2} \abs{x}^2, \quad  C > 1.
\]
The author believe that $u$ is a quadratic polynomial in this case.

The study of the complex Monge-Amp\`ere equation is largely motivated by the study of K\"ahler geometry. From the geometric point of view, our theorem implies the following rigidity statement.

\begin{cor}
\label{Cor geometric}
Suppose that $g$ is a Ricci-flat K\"ahler metric on $\C^n$. Let  $\varphi$ be its K\"ahler potential and $\mu_g$ be the induced measure. Denote $\abs{B_1}$ the Lebesgue measure of the unit ball.

If $\mu_g$ is comparable with Lebesgue measure, i.e., 
\[
C^{-1} dx \leq \mu_g \leq C dx
\] 
and
\begin{equation}
\label{cond growth-1}
\varphi = \smp{ \frac{\mu_g (B_1)}{\abs{B_1}} }^{1/n} \frac{\abs{x}^2}{2}+ o (\abs{x}^2) \quad \text{as }  x \rightarrow \infty,
\end{equation}
then $ \smp{ \frac{\abs{B_1}}{ \mu_g (B_1)} }^{1/n}g$ is the Euclidean metric.
\end{cor}

The above statement would be more satisfactory if one can replace the analytic condition \eqref{cond growth-1} by a pure geometric condition.

\bigskip

To end the introduction, we would like to mention that if one replace the condition \eqref{cond growth} in Theorem \ref{Thm Main} by
\begin{equation}
\label{cond growth-2}
u = \frac{1}{2} \abs{x}^2 + O(1)
\end{equation}
then the conclusion can be derived from an unpublished result of Kolodziej \cite{Kol}. This reference is pointed out to the author by S. Dinew. The method in this paper is independent from the work of Kolodziej.

\bigskip

\textbf{Acknowledgement.} The author is grateful to his thesis advisor Professor Duong Hong Phong for his inspirational discussions and helpful suggestions. The author also would like to thank Professor Ovidiu Savin, from whom the author have learned many important PDE techniques. The author also would like to thank Professor Valentino Tosatti and Professor Slawomir Dinew for many important discussions.

\section{Preliminaries}
In this section, we recall some basic facts regarding the complex Monge-Amp\`ere operator and the statement of Savin's small perturbation theorem.

\subsection{Complex Monge-Amp\`ere equation in real Hessian}
Let $\Sym(2n) $ be the space of $2n \times 2n$ symmetric matrices equipped with the standard spectral normal
\[
\norm{M} = \max \{ \abs{\lambda_i } \},  \lambda_i , \text{ eigenvalue of } M \in \Sym(2n).
\]
Let $\Herm (n)$ be the space of $n \times n$ Hermitian matrices. Denote the $n \times n$ identity matrix by $I_n$.

The space $\C^n$ can be identify to $\R^{2n}$ equipped with the complex structure 
\[
J = \begin{pmatrix}
0 & - I_n \\
I_n & 0
\end{pmatrix}.
\]
This identification induces an embedding $\imath$ of $\Herm (n )$ to $\Sym(2n)$
\[
\imath : H = A+i B \mapsto \begin{pmatrix}
A & -B \\
B & A
\end{pmatrix}.
\]
Moreover, we have
\[
\imath (\Herm (n )) = \bigp{ M \in \Sym(2n)  \; |\;  [M, J] = 0}.
\]

From now on, we shall identify a Hermitian matrix and its image under the embedding $\imath$.

Let $\varphi$ be a $C^2$-function on $\C^n$. Recall that
\[
\smp{i\ppb \varphi}^n =  \det (2  \varphi_{z_i z_{\bar{k}}} )  \; dx.
\]

Denote the real Hessian of $\varphi$ by $D^2 \varphi $. It is easy to see that
\[
\imath (2  \varphi_{z_i z_{\bar{k}}} ) =\frac{1}{2} \smp{D^2 \varphi  + J^T D^2\varphi  J}
\]
and
\[
\det (2  \varphi_{z_i z_{\bar{k}}} ) = \mathrm{det}^{1/2} \midp{\frac{1}{2} \smp{D^2 \varphi  + J^T D^2\varphi  J}}.
\]

The above discussion immediately implies the following lemma

\begin{lem}
Let the function $F :\Sym(2n) \rightarrow \R$ be given by
\begin{equation}
\label{eq function}
F (M):=\begin{cases}
 \mathrm{det}^{1/2} \midp{\frac{1}{2} \smp{M  + J^T M J} + I} -1 & M  + J^T M J \geq -I \\
-1 & \text{ otherwise. }
\end{cases}
\end{equation}

If $u$ is a viscosity solution of \eqref{eq CMA}, then 
\[
w := u  -\abs{x}^2/2
\]
is a viscosity solution of 
\[
F (D^2 w) = 0,\quad \text{ on } \R^{2n}.
\]
\end{lem}

\begin{proof}
Let $P$ be a quadratic polynomial that touches $w$ from above, then $P +\abs{x}^2/2$ touches $u$ from above. Since $u$ is a plurisubharmonic function, 
\[
\frac{1}{2} \smp{D^2 P  + J^T D^2P  J} \geq -I 
\]
Since $u$ is a viscosity subsolution of \eqref{eq CMA}, 
\[
 \mathrm{det}^{1/n} \midp{\frac{1}{2} \smp{D^2 \varphi  + J^T D^2\varphi  J} +I} =\det\midp{ 2 \smp{ P +\frac{\abs{x}^2}{2}}_{z_i z_{\bar{k}}}} \geq 1.
\]
Therefore, we conclude that $w$ is a viscosity subsolution of $F(D^2 w) =0$.

Let $P$ be a quadratic polynomial that touches $w$ from below. It suffices to consider the case
\[
\frac{1}{2} \smp{D^2 P  + J^T D^2P  J} > -I.
\]
Since otherwise, $F(D^2 P) = -1 < 0$. Again $P+\abs{x}^2/2$ touches $u$ from below. By the fact that $u$ is also a viscosity supersolution of \eqref{eq CMA}, we have
\[
 \mathrm{det}^{1/n} \midp{\frac{1}{2} \smp{D^2 \varphi  + J^T D^2\varphi  J} +I} =\det\midp{ 2 \smp{ P +\frac{\abs{x}^2}{2}}_{z_i z_{\bar{k}}}} \leq 1.
\] 
In turn $F(D^2 P ) \leq 0$ and $w$ is a viscosity supersolution of $F(D^2 w) = 0$. 
\end{proof}

\subsection{Small perturbation theorem} We recall the small perturbation theorem due to Savin \cite{Savin-smp}. The following definition will be convenient.

\begin{defn}
\label{defn Family}
Given constants $\delta , \theta, K>0$, the family $\mathcal{F}_{\delta ,\theta, K}$ consists of functions $F : \Sym(2n) \rightarrow \R$ that satisfy the following conditions:

\begin{description}
\item[H1] For every $M \in \Sym(2n)$
\[
F (M+ P ) \geq  F(M )  ,\quad  \forall  P \geq  0.
\]
\item[H2] $F (0)  = 0$.

\item[H3] For every $M$ with $\norm{M}\leq \delta$
\[
\theta^{-1} \norm{P} \geq  F(M+ P ) - F(M) \geq \theta \norm{P} , \forall P  \geq 0,  \norm{P} \leq \delta.
\]
\item[H4] $F$ is twice differentiable in the set $\{M \; |\; \norm{M}\leq \delta\}$ and 
\[
\abs{D^2 F (M )  } \leq K. 
\]
\end{description}
\end{defn}

We state the following version of the small perturbation theorem. 

\begin{thm}[Savin, 2007]
\label{Thm Savin-smp}
Given constant $\delta >0$, if
\begin{equation}
F \in \mathcal{F}_{\delta , \theta, K}, \text{ for some } \theta , K >0,
\end{equation}
then there exist constant $\mu$ only depending on $n, \delta, \theta, K$ such that, if $u \in C (B_1)$ is a viscosity solution of $F (D^2 u ) = 1$ and
\begin{equation} 
\norm{u }_{L^{\infty} (B_1) } \leq \mu,
\end{equation}
then $u$ is $C^2 (B_{1/2})$ and
\begin{equation}
\norm{D^2 u }_{L^{\infty} (B_{1/2} ) } \leq \delta.
\end{equation}
\end{thm}

We end this section with the following lemma.

\begin{lem}
\label{lem structure}
Let $F$ be the function on $\Sym(2n)$ given by \eqref{eq function}. There exists constants $\theta, K$ only depends on $n$ such that 
\[
F \in \mathcal{F}_{\delta, \theta, K}   \quad  \forall  \delta < 1/3.
\]
\end{lem}

\begin{proof}
The fact that $F$ satisfies \textbf{H1} and \textbf{H2} of Definition \ref{defn Family} follows immediately from its expression.

Let $D_{ij}F$ be the differentiation of $F$ with respect to the $ij$-entry of a matrix variable. By direct calculation, we have
\[
\begin{split}
& \abs{D_{ij} F (M)} = \abs{ \smp{ D_{ij} \mathrm{det}^{1/2}} \midp{\frac{1}{2} \smp{M  + J^T M  J}  +I} } \\
& \abs{D_{ij, kl}^2 F}= \abs{ \smp{ D_{ij，kl}^2 \mathrm{det}^{1/2}} \midp{\frac{1}{2} \smp{M + J^TM  J} +I } }.
\end{split}
\]
The fact that $F$ satisfies \textbf{H3} and \textbf{H4} then follows from 
\[
\norm{M +J^tM J } \leq 2 \norm{M}
\]
and
\[
\abs{D_{ij} \mathrm{det}^{1/2} (N)}, \abs{D_{ij, kl}^{1/2} \mathrm{det}^{1/2} (N)} \leq C (n) ,\quad \forall  N , \frac{1}{3}I \leq N \leq 3I.
\]
\end{proof}

\section{Proof of the main statements}
The Theorem \ref{Thm Main} follows from Theorem \ref{Thm Savin-smp} via a simple scaling argument.  

\begin{proof}[Proof of Theorem \ref{Thm Main}] By a translation of coordinate, it suffices to show that $u$ is $C^2$ and
\[
D^2 u (0) =  I.
\]

Consider the function
\begin{equation}
w_{R} (x) := \frac{1}{R^2} u (Rx ) - \frac{1}{2}\abs{x}^2
\end{equation}

Claim: 
\begin{equation}
\label{eq 1}
\norm{w_{R}}_{L^{\infty} (B_1) } \leq \frac{o(R^2)}{R^2}
\end{equation}

Consider the domain $B_1$, both $u(Rx)/R^2$ and $\abs{x}^2/2$ satisfies
\[
\smp{i\ppb \varphi }^n = dx 
\]
in $B_1$. By the comparison principle of the complex Monge-Amp\`ere operator, we conclude that
\[
\norm{ \frac{1}{R^2} u (Rx) - \frac{\abs{x}^2}{2} }_{L^{\infty}(B_1)} \leq \norm{  \frac{1}{R^2} u (Rx) - \frac{\abs{x}^2}{2} }_{L^{\infty} (\partial B_1 )} = \frac{1}{R^2} \norm{u -\abs{x}^2/2}_{L^{\infty}(\partial B_R ) }.
\]
The claim \eqref{eq 1} then follows from the assumption \eqref{cond growth}.

Now, let $F$ be the operator given by \eqref{eq function} and $\theta, K$ be the constants given by Lemma \ref{lem structure}.

For every $\delta \in (0, 1/3)$, let $\mu_{\delta}$ be the constant produced by Theorem \ref{Thm Savin-smp} with respect to $\mathcal{F}_{\delta, \theta, K}$.

By Lemma \ref{lem structure}, $F \in \mathcal{F}_{\delta,\theta, K}$. It is easy to see that
\[
w_{R} = \frac{1}{R^2} \smp{ u (Rx)  - \frac{1}{2} \abs{Rx}^2}
\]
satisfies $F (D^2 w)  = 0$ for any $R >0$. Moreover, by the claim \eqref{eq 1}, we can take $R$ large so that 
\[
\norm{w_R}_{L^{\infty}(B_1)} \leq \mu_{\delta}.
\]

Therefore, we can apply Theorem \ref{Thm Savin-smp} to conclude that $w_R$ is $C^2$ in $B_{1/2}$ and
\[
 \norm{ D^2 w_{R} (0 ) } \leq \delta.
\]
It follows then $u$ is $C^2$ in $B_{R/2}$ and 
\[
\norm{D^2 u  (0)  - I} \leq \delta.
\]
The desired conclusion follows by letting $\delta$ tend to $0$.
\end{proof}

Nest, we prove Corollary \ref{Cor geometric}.

\begin{proof} [Proof of Corollary \ref{Cor geometric}] Since $g$ is Ricci flat, we have
\[
\Delta \log  \det (\varphi_{z_i z_{\bar{k}}}) = 0, \text{ on } \C^n
\]
Since $\mu_g$ is comparable with Lebesgue measure, we have
\[
C^{-1} \leq   \det (\varphi_{z_i z_{\bar{k}}})  \leq C.
\]
Therefore, $ \log  \det (\varphi_{z_i z_{\bar{k}}}) $ is a bounded harmonic function on entire $\C^n$. Henceforth, 
\[
 \log  \det (\varphi_{z_i z_{\bar{k}}})  = \text{constant}.
\]

It follows then $\varphi $ satisfies 
\[
\smp{i\ppb \varphi} =\frac{\mu_g (B_1)}{\abs{B_1}} dx, \textit{ on } \C^n
\] 
Along with the assumption \eqref{cond growth-1}, we see
\[
\tilde{\varphi}:=  \smp{ \frac{\abs{B_1}}{\mu_g (B_1)} }^{1/n} \varphi
\]
satisfies the hypotheses of Theorem \ref{Thm Main}. Therefore, we conclude that
\[
\tilde{\varphi} = \frac{1}{2}\abs{x}^2 +  l (x).
\]
The desired conclusion follows.
\end{proof}


\end{document}